\documentclass{article}

\usepackage[english]{babel}
\usepackage{setspace}
\usepackage[letterpaper,top=2cm,bottom=2cm,left=3cm,right=3cm,marginparwidth=1.75cm]{geometry}

\usepackage[utf8]{inputenc}
\usepackage{amsfonts}
\usepackage{amsthm}
\usepackage{amsmath}
\usepackage{xcolor}
\usepackage{mathrsfs}
\usepackage{csquotes}
\usepackage{mathtools}
\usepackage{amsmath}
\usepackage{amsthm}
\usepackage{graphicx}
\usepackage{amsmath,amssymb}
\usepackage[colorlinks=true, allcolors=blue]{hyperref}

\newtheorem{theorem}{Theorem}
\newtheorem*{theorem*}{Theorem}
\numberwithin{theorem}{section} % number theorems inside sections

\numberwithin{proposition}{section} % number propositions inside sections

\theoremstyle{remark}

\newcommand{\Eps}{\mathcal{F}}
\newcommand{\T}{\mathcal{T}}

\usepackage[style=ieee, sorting=nyt]{biblatex}
\title{Rainbow Triangles in Families of Triangles}
\author{Ido Goorevitch \, and \, Ron Holzman\\Department of Mathematics\\Technion-Israel Institute of Technology\\\texttt{ido.g@campus.technion.ac.il}, \, \texttt{holzman@technion.ac.il}}

\date{}

\begin{document}
\maketitle

\begin{abstract}
    We prove that a family $\T$ of distinct triangles on $n$ given vertices that does not have a rainbow triangle (that is, three edges, each taken from a different triangle in $\T$, that form together a triangle) must be of size at most $\frac{n^2}{8}$. We also show that this result is sharp and characterize the extremal case. In addition, we discuss a version of this problem in which the triangles are not necessarily distinct, and show that in this case, the same bound holds asymptotically. After posting the original arXiv version of this paper, we learned that the sharp upper bound of $\frac{n^2}{8}$ was proved much earlier by Gy\H{o}ri (2006) and independently by Frankl, F\"uredi and Simonyi (unpublished). Gy\H{o}ri also obtained a stronger version of our result for the case when repetitions are allowed.
\end{abstract}

%===============================/
%========== SECTION 1 ==========/
%===============================/

\section{Introduction}
    Let $\Eps$ be a family of sets. A rainbow set (with respect to $\Eps$) is a subset $R\subseteq\cup\Eps$, together with an injection $\sigma:R\to\Eps$ such that $e\in\sigma(e)$ for all $e\in R$. We view every member of $\Eps$ as a different color, and every $e\in R$ as colored by $\sigma(e)$, hence the ``rainbow" terminology.
    % In general, the family $\Eps$ is allowed to have repeated members. Nevertheless, in section 2 we discuss a rainbow problem, in which $\Eps$ must be a set.
    Note that, in general, we use the term ``family" in the sense of ``multiset", allowing repeated members and treating them as different colors. 

    If each set in $\Eps$ satisfies some property $\mathcal{P}$ of interest, how large does $\Eps$ need to be in order to guarantee the existence of a rainbow set $R$ that also satisfies $\mathcal{P}$? A classic result of this kind is B{\'a}r{\'a}ny's colorful Carath{\'e}odory theorem \cite{barany1982Caratheodory}: Every family of $n+1$ subsets of $\mathbb{R}^n$, each containing a point $a$ in its convex hull, has a rainbow set satisfying the same property.
    There are several results of this type in extremal graph theory. For example, generalizing a theorem of Drisko \cite{drisko1998transversals}, Aharoni and Berger \cite{aharoniberger2009rainbowmatchings}  proved that $2n - 1$ matchings of size $n$ in any bipartite graph have a rainbow matching of size $n$. Another example, and the motivation for this paper, is a theorem by Aharoni et al.~\cite{aharoni2021rainbowoddcycles}: Every family of $n$ odd cycles (viewed as edge sets) on $n$ vertices has a rainbow odd cycle. 
    
    Here we are interested in the case when both the cycles in the given family and the desired rainbow one are not just any odd-length cycles but triangles. The existence of rainbow triangles has been studied in the literature in a different context: Given a graph and a coloring of its edges, under what conditions must there be a triangle whose edges have distinct colors? A fundamental structural result is due to Gallai \cite{gallai1967transitiv}:
    If $G$ is an edge-colored complete graph on at least two vertices without a rainbow triangle, there is a nontrivial partition $\mathscr{P}$ of $V(G)$ such that the edges between different parts are colored with at most two colors, and the edges between each pair of parts $P,Q\in\mathscr{P}$ all have the same color. Erd\H{o}s et al.~\cite{erdos1975anti} observed that every coloring of the edges of the complete graph $K_n$ with at least $n$ colors has a rainbow triangle; this was the starting point of their ``anti-Ramsey theory". Several more detailed conditions guaranteeing rainbow triangles were found in \cite{gyarfas2004tricolored,gyarfas2010noncomplete,li2012hetero,li2013rainbowc3c4,li2014rainbowtriangles,aharoni2019caccetta} and other studies.

    % GRAPH
    \begin{figure}
    \centering
    \includegraphics[width=0.3\textwidth]{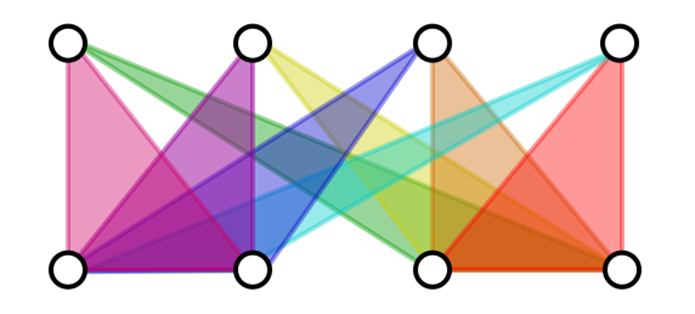}
    \caption{\label{fig:graphexample1}The family $\T^*_8$.}
    \end{figure}

    Let us return to our problem of determining the largest size of a family of triangles on $n$ vertices without a rainbow triangle. A cheap upper bound of $\frac{n^2}{4}$ can be
    proved by combining Mantel's theorem and the general (and easy) Proposition~3.1 in~\cite{aharoni2021rainbowoddcycles} concerning rainbow sets. 
    A lower bound of $\frac{n^2}{8}$ can be achieved by taking $\frac{n}{4}$ disjoint pairs of vertices, and connecting each pair to each of the remaining $\frac{n}{2}$ vertices with a triangle. We denote this family by $\T^*_n$ (see Figure 1). These observations indicate that the answer to our question about triangles is quadratic in $n$, in contrast to the answer to the similar problem about arbitrary odd cycles being linear in $n$, as mentioned above. 
    
    Determining the correct constant between $\frac{1}{8}$ and $\frac{1}{4}$ requires more effort. In Section 2, using a method inspired by a folkloric proof of Mantel's theorem~\cite{aigner1995turan}, we prove that $\frac{n^2}{8}$ is the right answer when the family consists of distinct triangles. Moreover, we show that $\T^*_n$ is the only extremal such family. In Section 3 we use the result of Ruzsa and Szemer{\'e}di \cite{ruzsa1978triple} about the famous $(6,3)$-problem, to show that if triangles are allowed to repeat, the same upper bound holds asymptotically. Namely, when counting triangles with repetitions, we still cannot have more than $\frac{n^2}{8}(1+o(1))$ of them. As mentioned in the abstract, we have recently learned that our results (except for the characterization of the extremal family) had been discovered earlier.

%===============================/
%========== SECTION 2 ==========/
%===============================/

\section{Families of distinct triangles}

In this section, the family $\T$ of triangles will be a set. Given $n$ vertices, with $n$ divisible by $4$, the family $\T^*_n$ introduced above is a set of $\frac{n^2}{8}$ triangles without a rainbow triangle. The following theorem says that this is the largest possible size of such a family. It turns out to be a rediscovery of Theorem~2 in Gy\H{o}ri~\cite{gyori2006}, which was also independently proved by Frankl, F\"uredi and Simonyi (unpublished).

\begin{theorem}
Let $V$ be a set of $n$ vertices and $\T$ a set of triangles on $V$ having no rainbow triangle. Then $|\T| \le \frac{n^2}{8}$.
\end{theorem}

\begin{proof}
Let $V$ and $\T$ be as in the statement of the theorem. Consider the graph $G$ on the vertex set $V$ having the union of the edge sets of the triangles in $\T$ as its edge set.
Let $A\subseteq V$ be a largest independent set of vertices in $G$.
Let $B = V\setminus A$, and denote by $E(B)$ the set of edges in the subgraph induced by $G$ on $B$.
Note that, by the independence of $A$, each triangle in $\T$ must have at least one edge in $E(B)$. Also note that, by the absence of rainbow triangles, each triangle in $\T$ has at most one edge that it shares with other triangles in $\T$.

We define a function $\beta:\T\longrightarrow E(B)$ as follows. For each $t\in \T$, we set $\beta(t)$ to be one of its edges that is in $E(B)$; if $t$ has more than one edge (and hence all its edges) in $E(B)$, and one of them is shared with another triangle, we choose that edge to be $\beta(t)$; otherwise, $\beta(t)$ is chosen arbitrarily.
For an edge $e\in E(B)$, we denote $d(e)\vcentcolon= |\beta^{-1}(e)|$.
Then, we have
    \begin{equation}
        \sum_{b\in B}{\sum_{e\in E(B): b\in e}{d(e)}} = 2|\T|.
    \end{equation}
Indeed, each triangle in $\T$ corresponds to one edge in $E(B)$ under $\beta$, and thus it is counted once in $d(e)$ for a single $e\in E(B)$. Each edge $e=\{u,v\}\in E(B)$ is considered twice in the inner sum - once with respect to $u$ and once with respect to $v$, and so we get $2|\T|$ in total. 

We will show below that for each $b\in B$,
    \begin{equation}
        \sum_{e\in E(B): b\in e}{d(e)} \leq |A|.
    \end{equation}
By summing over all $b\in B$, this implies that
    \begin{equation}
        \sum_{b\in B}{\sum_{e\in E(B): b\in e}{d(e)}} \leq |B||A|.
    \end{equation}
From (1) and (3) we get what we need:
    \begin{equation}
        2|\T| \leq |B||A|\leq (\frac{|A|+|B|}{2})^2 = \frac{n^2}{4} \Rightarrow |\T| \leq \frac{n^2}{8}.
    \end{equation}

It remains to show (2). Fix a vertex $b\in B$. It is enough to exhibit an independent set $I^b \subseteq V$ of size $\sum_{e\in E(B): b\in e}{d(e)}$, and from $A$'s maximality, the desired inequality will follow.
We construct the independent set $I^b$ as follows.

For each edge $e\in E(B)$ such that $b\in e$, and for each triangle $t\in\beta^{-1}(e)$, we pick a vertex of $t$ according to the rules below, and place it in $I^b$:
\begin{itemize}
\item If $\beta^{-1}(e) = \{t\}$, we pick the vertex of $e = \beta(t)$ that is not $b$.
\item Else, we pick the vertex of $t$ that is not in $e = \beta(t)$.
\end{itemize}

% GRAPH
\begin{figure}
\centering
\includegraphics[width=0.25\textwidth]{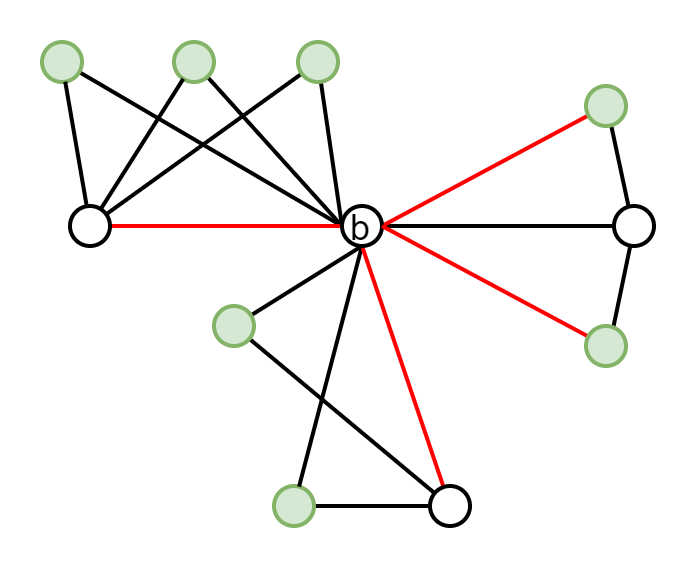}
\caption{\label{fig:graphexample2}The set $I^b$ (in green) for a given $b$. In each triangle $t$, the edge $\beta(t)$ is marked in red.}
\end{figure}

See Figure 2 for an illustration. We have to make sure that:
\begin{itemize}
    \item The vertices chosen that way are distinct (and thus indeed $|I^b|=\sum_{e\in E(B): b\in e}{d(e)}$). 
    \\Suppose for the sake of contradiction that two distinct triangles $t_1, t_2$ contribute the same vertex $v$ to $I^b$. Then $t_1, t_2$ share the edge $\{b,v\}$, so we have $V(t_1)=\{b,v,v_1\}$ and $V(t_2)=\{b,v,v_2\}$ with $v_1\neq v_2$. 
    If $v\notin B$ then $\beta(t_1)=\{b,v_1\}$. Because $t_1$ shares $\{b,v\}$ with another triangle, $\{b,v_1\}$ is not shared with another triangle, and so the first case in the definition of $I^b$ applies: $t_1$ contributes $v_1$ to $I^b$, not $v$ as assumed. If $v \in B$ then, by the definition of $\beta$, we have $\beta(t_1)=\beta(t_2)=\{b,v\}$. Thus, the second case in the definition of $I^b$ applies: $t_1$ contributes $v_1$ to $I^b$, again contradicting our assumption.
    
    \item $I^b$ is independent. \\Suppose for the sake of contradiction that two vertices $v_1\neq v_2\in I^b$ that were chosen from triangles $t_1,t_2$ respectively 
    form an edge $\{v_1,v_2\}$, taken from some triangle $t\in \T$. If $t\neq t_1, t_2$, a rainbow triangle arises on the vertices $b, v_1, v_2$. So, we may assume w.l.o.g. that $t=t_1$, i.e., $V(t_1)=\{b,v_1,v_2\}$ and $V(t_2)=\{b,v_2,v_3\}$ for some $v_3\neq v_1$. Now, an argument similar to the one in the previous paragraph shows that, whether or not $v_2$ belongs to $B$, the vertex contributed by $t_2$ to $I^b$ is $v_3$, not $v_2$ as assumed.
\end{itemize}

That finishes the proof.
\end{proof}

We now characterize the equality case. That is, families $\T$ of exactly $\frac{n^2}{8}$ distinct triangles on $n$ vertices without a rainbow triangle. Such a construction must in particular have $4|n$, and the following theorem shows that it must coincide with the family $\T^*_n$.

\begin{theorem}
Let $4|n\in\mathbb{N}$. Let $\T^*_n$ be the family of triangles constructed by taking $\frac{n}{4}$ disjoint pairs of vertices, and connecting each of them to each of the remaining $\frac{n}{2}$ vertices with a triangle.
Then $\T^*_n$ is the unique set of $\frac{n^2}{8}$ triangles on $n$ vertices without a rainbow triangle, up to isomorphism.
\end{theorem}

\begin{proof}
Let $\T$ be a set of $\frac{n^2}{8}$ triangles on $n$ vertices without a rainbow triangle. We will show that $\T \cong \T^*_n$. To this end, we refer to the notations and arguments in the proof of Theorem 2.1, where we showed: 
\begin{equation*}
        2|\T| = \sum_{b\in B} \sum_{e\in E(B): b\in e}{d(e)} \leq \sum_{b\in B} |A| = |B||A|\leq (\frac{|A|+|B|}{2})^2 = \frac{n^2}{4}. 
    \end{equation*}
    By our assumption, $2|\T| = \frac{n^2}{4}$ and thus all the inequalities are equalities. Our proof proceeds in two steps.

\textbf{Step 1.} 
    We show that $\T$ does not have a triangle with all three vertices in $B$.

    Assume for the sake of contradiction that $t$ is a triangle in $\T$ contained in $B$. We denote by $b\in B$ the vertex of $t$ that is not in $\beta(t)=\{u,v\}$. Thus $V(t)=\{b,u,v\}$.
    Then, both $\{b,u\}$ and $\{b,v\}$ are not shared with other triangles (otherwise by our choice of $\beta(t)$, we would have picked one of them instead of $\{u,v\}$). We claim that $v \notin I^b$ and $I^b\cup \{v\}$ is independent, where $I^b$ is the independent set of vertices we chose for the vertex $b$ in the proof of Theorem 2.1. This will yield the desired contradiction for Step~1, because equality in the inequalities (2) gives in particular $|I^b| = \sum_{e\in E(B): b\in e}d(e) = |A|$, and $A$ is a largest independent set.
    \begin{itemize}
        \item First, we show that $v \notin I^b$. \\Note that $t$ itself does not contribute a vertex to $I^b$, because $\beta(t)$ does not contain $b$. Thus, if $v \in I^b$ then it must be contributed by some triangle $t' \ne t$. But then $t$ and $t'$ share the edge $\{b,v\}$, contradicting the above.
        \item The set $I^b\cup \{v\}$ is independent.\\ $I^b$ itself is independent and so it remains to rule out an edge of the form $\{v,v'\}$ where $v'\in I^b$. Suppose that $\{v,v'\}$ is such an edge. Then $v'$ is contributed to $I^b$ by some triangle $t'\ne t$, and the edge $\{v,v'\}$ belongs to some triangle $t''$. Note that $v' \ne u$ (or else $t$ and $t'$ would share the edge $\{b,u\}$) and hence $t'' \ne t$. Also $t'' \ne t'$, otherwise their common vertex set would be $\{b,v,v'\}$, and this triangle shares the edge $\{b,v\}$ with $t$. Thus, $t$, $t'$ and $t''$ are three distinct triangles which form together a rainbow triangle on $\{b,v,v'\}$, which is forbidden.
    \end{itemize}

% GRAPH
\begin{figure}
\centering
\includegraphics[width=0.3\textwidth]{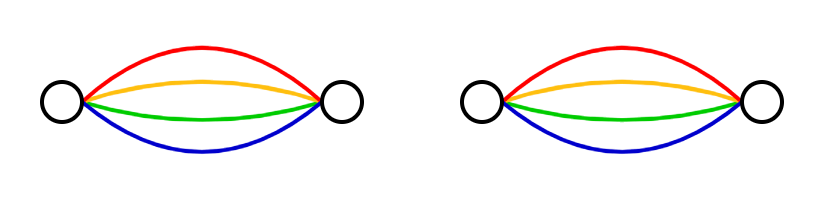}
\caption{\label{fig:graphexample3}The colored multigraph $(\T^*_8)^B$.}
\end{figure}

\textbf{Step 2.} We show that $\T$ is of the form $\T^*_n$.

Writing $m\vcentcolon=\frac{n}{2}$, we note that equality in the inequalities above requires that $|A|=|B|=m$.
    From Step 1, we know that every triangle in $\T$ must have one vertex in $A$ and two vertices in $B$. 
    Now, we look at the colored multigraph $\T^B$ induced by $\T$ on $B$ by taking from each triangle its only edge in $B$, and coloring it with its only vertex in $A$. We say that an edge in $\T^B$ is simple if it has no parallel edges, and otherwise it is non-simple.
	So $\T^B$ has $m$ vertices and its edges are colored with $m$ colors. In addition, it has the following properties:
	\begin{enumerate}
    \item $\T^B$ is triangle-free.
    \item If $e_1,e_2,e_3$ is a path of three edges, then $e_1$ and $e_3$ have distinct colors.
	\item If $e_1,e_2$ are two edges sharing a vertex and having the same color, then both $e_1$ and $e_2$ are simple.
	\item $\T^B$ is $m$-regular.
	\end{enumerate}
	Properties 1-3 hold due to the fact that $\T$ is rainbow-triangle-free (actually, one can check that 1-3 are also enough to preclude rainbow triangles in $\T$). Property 4 holds because equality in (2) means that for each $b\in B$ we have $\sum_{e\in E(B): b \in e}d(e)=m$, and this counts exactly the number of edges in $\T^B$ incident with $b$.
	
    We will show below that a multigraph with these properties must consist of $\frac{m}{2}$ disjoint pairs of vertices, each having $m$ edges of all colors between them (see Figure 3).
    Such $\T^B$ corresponds to a family $\T\cong \T^*_n$, which concludes the proof.
    
	First, if every edge in $\T^B$ is non-simple, then by property 3, every two edges of the same color are disjoint.
	That is, the edges of each color $a$ form a matching $M_a$ in $B$. Moreover, property 4 implies that every vertex in $B$ is incident with edges of all $m$ colors, so the $M_a$'s are perfect matchings. Finally, due to property 2, all the $M_a$'s must induce the same partition of $B$ into $\frac{m}{2}$ pairs of vertices, which gives us the desired multigraph.
	
	Thus, we may assume that there is a simple edge $e=\{v,v'\}$ in $\T^B$. Suppose there are $l$ simple edges incident with $v$ (excluding $e$) and $k$ non-simple ones. Similarly, there are $l'$ and $k'$ simple and non-simple edges respectively incident with $v'$ (see Figure 4).
	Now, from property 4 applied to both $v$ and $v'$ we have
	\begin{equation*}
	    k+l+1 = k'+l'+1 = m.
	\end{equation*}
	Note that $k>0$, because otherwise $v$ has $m$ distinct neighbors, which is impossible as $|B|=m$.
	Symmetrically, $k'>0$.
	Therefore, there are at least $l+1$ and $l'+1$ vertices in the two respective sides of $e$, all different due to property 1. Along with ${v,v'}$, there are at least $l+l'+4$ vertices in $B$ and thus
	\begin{equation*}
	    l+l'+4\leq m.
	\end{equation*}
	Finally, property 2 implies that the edges of the two sides have different colors. Due to property 3, in the two sides there are at least $k,k'$ colors respectively. This gives
	\begin{equation*}
	    k+k'\leq m.
	\end{equation*}
	By summing the last two inequalities, we obtain
	\begin{equation*}
	    (k+k')+(l+l'+4)\leq 2m.
	\end{equation*}
	This can be rearranged as
	
	\begin{equation*}
	    (k+l+1)+(k'+l'+1)\leq 2m-2,
	\end{equation*}
    which is a contradiction because the left-hand side equals $2m$.
\end{proof}

	% GRAPH
\begin{figure}
\centering
\includegraphics[width=0.35\textwidth]{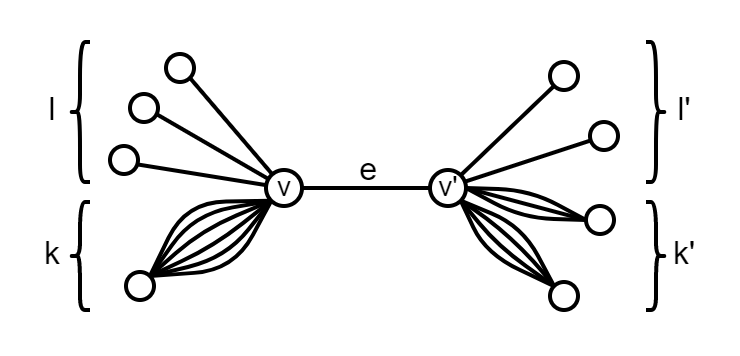}
\caption{\label{fig:graphexample4}The notations in the proof of Step 2.}
\end{figure}

%===============================/
%========== SECTION 3 ==========/
%===============================/

\section{Families of not necessarily distinct triangles}

We now treat another version of this problem, where the triangles are not necessarily distinct.
Suppose $\T$ is a rainbow-triangle-free family of triangles on $n$ vertices, where duplicates are allowed. We want to bound the size of $\T$ as a multiset. 

Of course, each triangle must appear at most twice in $\T$, as three copies of the same triangle create a rainbow triangle. Therefore, the triangles in $\T$ can be of two types - those that appear twice in $\T$, and those that appear only once. Let $\T_1$ be the set of all triangles in $\T$ (one copy of each), and let $\T_2 \subseteq \T_1$ be the set of those triangles that appear twice. Then $|\T|=|\T_1|+|\T_2|$ and we can bound each summand separately.

First, from Theorem 2.1 it follows that $|\T_1|\leq \frac{n^2}{8}$. To bound $|\T_2|$, we observe that the triangles in $\T_2$ are edge-disjoint, due to the absence of rainbow triangles. For the same reason, the graph $G_2$ whose edge set is the disjoint union of the edge sets of the triangles in $\T_2$ has no triangles other than those in $\T_2$. Thus, in $G_2$ every edge belongs to a unique triangle, and we can apply the following version of a theorem of Ruzsa and Szemerédi \cite{ruzsa1978triple}:
\begin{theorem*}[Ruzsa and Szemerédi]
    A graph $G$ on $n$ vertices in which every edge belongs to a unique triangle, has $o(n^2)$ edges.
\end{theorem*}

We note that Ruzsa and Szemer{\'e}di treated the equivalent $(6,3)$-problem, which asks for the maximal number of triples of points that one can select from $n$ given points, in such a way that no six points contain three of the selected triples. A stronger form of the bound, namely $\frac{n^2}{e^{\Omega(\log^*n)}}$, was proved by Fox~\cite{fox2011removallemma}.

We observe that the graph $G_2$ has $3|\T_2|$ edges, so we get $|\T_2|=o(n^2)$.
Together with the bound on $|\T_1|$ we obtain $|\T_1|+|\T_2|\leq \frac{n^2}{8}(1+o(1))$. Thus we conclude the following:

\begin{theorem}
    Let $V$ be a set of $n$ vertices and $\T$ a family of (not necessarily distinct) triangles on $V$ having no rainbow triangle. Then $|\T|\leq\frac{n^2}{8}(1+o(1))$.
\end{theorem}

We remark that for small values of $n$, there do exist constructions of families of triangles with repetitions having no rainbow triangle, which beat the $\frac{n^2}{8}$ bound. An example for $n=9$ is shown in Figure~5. We conjecture, however, that for sufficiently large $n$, the $\frac{n^2}{8}$ bound holds in its exact form even if repetitions are allowed. As it turns out, Theorem~1 in Gy\H{o}ri~\cite{gyori2006} established this, in a more general form, for $n \ge 100$.

% GRAPH
\begin{figure}
\centering
\includegraphics[width=0.18\textwidth]{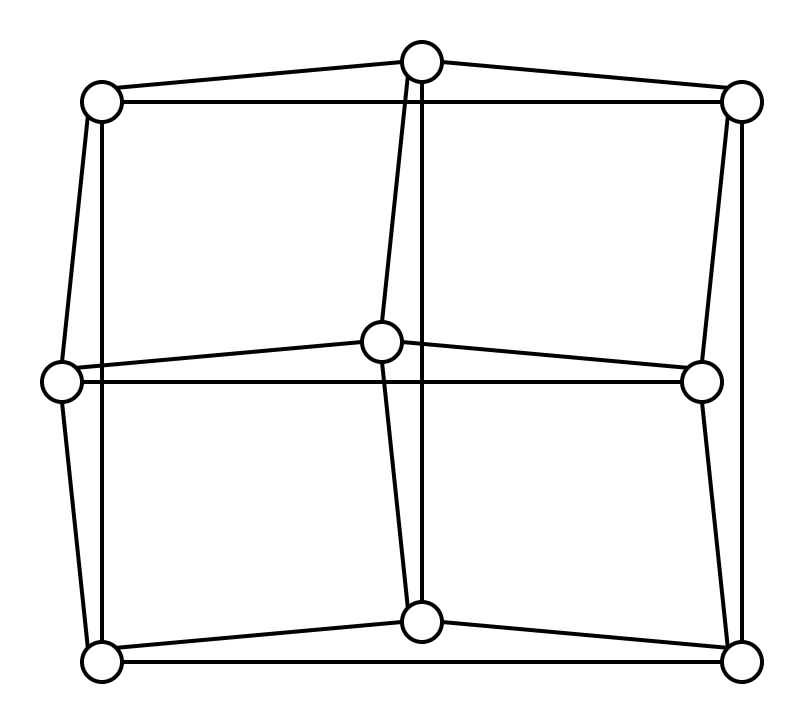}
\caption{\label{fig:graphexample5}A family of $12$ triangles on $9$ vertices, built of two copies of each of the $6$ triangles in the drawing, without a rainbow triangle.}
\end{figure}

% ACKNOWLEDGEMENTS
\subsection*{Acknowledgment}
We thank Gonen Gazit for his help in creating computer-generated explicit solutions for small values of $n$, which provided us with hints for the general case. We are grateful to Zolt\'an F\"uredi for drawing our attention to~\cite{gyori2006}.

% BIBLIOGRAPHY
\printbibliography

@article{aigner1995turan,
  title={Tur{\'a}n's graph theorem},
  author={Aigner, Martin},
  journal={The American Mathematical Monthly},
  volume={102},
  number={9},
  pages={808--816},
  year={1995},
  publisher={Taylor \& Francis}
}

@article{ruzsa1978triple,
  title={Triple systems with no six points carrying three triangles},
  author={Ruzsa, Imre Z and Szemer{\'e}di, Endre},
  journal={Combinatorics (Keszthely, 1976), Coll. Math. Soc. J. Bolyai},
  volume={18},
  pages={939--945},
  year={1978}
}

@article{barany1982Caratheodory,
  title={A generalization of {C}arath{\'e}odory's theorem},
  author={B{\'a}r{\'a}ny, Imre},
  journal={Discrete Mathematics},
  volume={40},
  number={2-3},
  pages={141--152},
  year={1982},
  publisher={Elsevier}
}

@article{aharoni2021rainbowoddcycles,
  title={Rainbow odd cycles},
  author={Aharoni, Ron and Briggs, Joseph and Holzman, Ron and Jiang, Zilin},
  journal={SIAM Journal on Discrete Mathematics},
  volume={35},
  number={4},
  pages={2293--2303},
  year={2021},
  publisher={SIAM}
}

@article{li2014rainbowtriangles,
  title={Rainbow triangles in edge-colored graphs},
  author={Li, Binlong and Ning, Bo and Xu, Chuandong and Zhang, Shenggui},
  journal={European Journal of Combinatorics},
  volume={36},
  pages={453--459},
  year={2014},
  publisher={Elsevier}
}

@article{li2013rainbowc3c4,
  title={Rainbow {C}3’s and {C}4’s in edge-colored graphs},
  author={Li, Hao},
  journal={Discrete Mathematics},
  volume={313},
  number={19},
  pages={1893--1896},
  year={2013},
  publisher={Elsevier}
}

@article{aharoniberger2009rainbowmatchings,
  title={Rainbow matchings in $ r $-partite $ r $-graphs},
  author={Aharoni, Ron and Berger, Eli},
  journal={The Electronic Journal of Combinatorics},
  volume={16},
  pages={\#R119},
  year={2009}
}

@article{fox2011removallemma,
  title={A new proof of the graph removal lemma},
  author={Fox, Jacob},
  journal={Annals of Mathematics},
  volume={174},
  number={1},
  pages={561--579},
  year={2011},
  publisher={JSTOR}
}

@article{drisko1998transversals,
  title={Transversals in row-Latin rectangles},
  author={Drisko, Arthur A},
  journal={Journal of Combinatorial Theory, Series A},
  volume={84},
  number={2},
  pages={181--195},
  year={1998},
  publisher={Elsevier}
}

@article{gallai1967transitiv,
  title={Transitiv orientierbare {G}raphen},
  author={Gallai, Tibor},
  journal={Acta Mathematica Hungarica},
  volume={18},
  number={1-2},
  pages={25--66},
  year={1967},
  publisher={Akad{\'e}miai Kiad{\'o}, co-published with Springer Science+ Business Media BV~…}
}

@article{aharoni2019caccetta,
  title={Rainbow triangles and the {C}accetta-{H}{\"a}ggkvist conjecture},
  author={Aharoni, Ron and DeVos, Matthew and Holzman, Ron},
  journal={Journal of Graph Theory},
  volume={92},
  number={4},
  pages={347--360},
  year={2019},
  publisher={Wiley Online Library}
}

@article {gyarfas2004tricolored,
    AUTHOR = {Gy\'{a}rf\'{a}s, Andr\'{a}s and Simonyi, G\'{a}bor},
     TITLE = {Edge colorings of complete graphs without tricolored
              triangles},
   JOURNAL = {Journal of Graph Theory},
  FJOURNAL = {Journal of Graph Theory},
    VOLUME = {46},
      YEAR = {2004},
    NUMBER = {3},
     PAGES = {211--216},
      }

@article {gyarfas2010noncomplete,
    AUTHOR = {Gy\'{a}rf\'{a}s, Andr\'{a}s and S\'{a}rk\"{o}zy, G\'{a}bor N.},
     TITLE = {Gallai colorings of non-complete graphs},
   JOURNAL = {Discrete Mathematics},
  FJOURNAL = {Discrete Mathematics},
    VOLUME = {310},
      YEAR = {2010},
    NUMBER = {5},
     PAGES = {977--980},
      }

@incollection {erdos1975anti,
    AUTHOR = {Erd\H{o}s, P. and Simonovits, M. and S\'{o}s, V. T.},
     TITLE = {Anti-{R}amsey theorems},
 BOOKTITLE = {Infinite and finite sets ({C}olloq., {K}eszthely, 1973;
              dedicated to {P}. {E}rd\H{o}s on his 60th birthday), {V}ol. {II}},
     PAGES = {633--643},
 PUBLISHER = {North-Holland, Amsterdam},
      YEAR = {1975},
   MRCLASS = {05C15},
  MRNUMBER = {0379258},
MRREVIEWER = {Rudolf Halin},
}

@article {li2012hetero,
    AUTHOR = {Li, Hao and Wang, Guanghui},
     TITLE = {Color degree and heterochromatic cycles in edge-colored
              graphs},
   JOURNAL = {European Journal of Combinatorics},
  FJOURNAL = {European Journal of Combinatorics},
    VOLUME = {33},
      YEAR = {2012},
    NUMBER = {8},
     PAGES = {1958--1964},
      }

@article {gyori2006,
    AUTHOR = {Gy\H{o}ri, Ervin},
     TITLE = {Triangle-free hypergraphs},
   JOURNAL = {Combin. Probab. Comput.},
  FJOURNAL = {Combinatorics, Probability and Computing},
    VOLUME = {15},
      YEAR = {2006},
    NUMBER = {1-2},
     PAGES = {185--191},
         MRCLASS = {05C35 (05C65)},
  MRNUMBER = {2195581},
MRREVIEWER = {J\'{o}zsef Balogh},
       }

\end{document}